\documentclass{amsproc}

\usepackage[colorlinks = true, linkcolor = black, anchorcolor = ., citecolor = black,filecolor=black,menucolor=black,runcolor=black,urlcolor=black]{hyperref}

\newtheorem{theorem}{Theorem}[section]
\newtheorem{lemma}[theorem]{Lemma}
\newtheorem{proposition}[theorem]{Proposition}

\theoremstyle{definition}

\theoremstyle{remark}
\newtheorem{remark}[theorem]{Remark}

\numberwithin{equation}{section}

\newcommand{\R}{\mathbb{R}}
\newcommand{\eps}{\varepsilon}
\newcommand{\epsi}{\varepsilon^{-1}}

\newcommand{\F}{\mathcal{F}}
\newcommand{\e}{\normalfont{\text{e}}}

\usepackage[maxbibnames=99,backend=bibtex]{biblatex}

\begin{document}

\title{Infinite dimensional Slow Manifolds for a Linear Fast-Reaction System}

\author{Christian Kuehn}
\address{Department of Mathematics, Technical University of Munich}
\email{ckuehn@ma.tum.de}
\thanks{The first author was supported in part by DFG grant 456754695 and a Lichtenberg professorship via the  VolkswagenStiftung. The third author was supported by DFG grant 456754695.}

\author{Pascal Lehner}
\address{Department of Mathematics, University of Klagenfurt}
\email{pascal.lehner@aau.at}
\author{Jan-Eric Sulzbach}
\address{Department of Mathematics, Technical University of Munich}
\email{janeric.sulzbach@ma.tum.de}




\begin{abstract}
The aim of this expository paper is twofold. We start with a concise overview of the theory of invariant slow manifolds for fast-slow dynamical systems starting with the work by Tikhonov and Fenichel to the most recent works on infinite-dimensional fast-slow systems. The main part focuses on a class of linear fast-reaction PDE, which are particular forms of fast-reaction systems. The first result shows the convergence of solutions of the linear system to the limit system as the time-scale parameter $\eps$ goes to zero. Moreover, from the explicit solutions the slow manifold is constructed and the convergence to the critical manifold is proven. The subsequent result, then, states a generalized version of the Fenichel-Tikhonov theorem for linear fast-reaction systems.
\end{abstract}

\maketitle



\section{Introduction}

In all natural sciences, one frequently observes dynamical systems that exhibit a multiple timescale behaviour. The fundamental reason for this is that natural systems tend to be influenced by many different processes, and the more processes are taken into account, the more likely it is to obtain systems that evolve on different timescales. The involved timescales can range, depending on the application, from nanoseconds, for example, due to fluctuations on atomic levels to years due to, for example, seasonal variability of the earth. The modeling of such systems often results in a coupled system of differential equations with many variables. These system can be computationally challenging, especially if the system is resolved on the fastest timescale. \\

One building block to model these multi-scale systems is to consider the case two timescales, so-called fast-slow systems. These differential equations can be written in their standard form abstractly as 
\begin{align}\label{general eq eps}
    \begin{split}
        \eps \partial_t u^{\eps} &= F(u^\eps, v^\eps, \eps), \\
             \partial_t v^{\eps} &= G(u^\eps, v^\eps, \eps),
    \end{split}
\end{align}
where $u^\eps=u^\eps(t) \in X $ is the fast variable and $v^\eps=v^\eps(t)\in Y$ the slow variable. The variables depend on time $t$ belonging to some interval $\mathcal{I}$ with $0 \in \mathcal{I}$. The state spaces $X$ and $Y$ are often assumed to be Banach spaces, while the system \eqref{general eq eps} is an ODE if state spaces are finite dimensional, while PDEs are the most common examples if $X$ and/or $Y$ are infinite dimensional. With $F,G$ we denote general maps, that may contain (differential) operators, which determine the evolution of the variables. The time scale parameter $\eps>0$ formally determines the quantitative amount of time scale separation between the two variables. Furthermore, suitable initial conditions can be added to the fast-slow system. We write the variables $u^\eps,v^\eps$ with a superscript $\eps$ to indicate the dependence of solutions on $\eps$.

Since in many applications $\eps$ is assumed to be very small, it is natural to consider the limit $\eps \to 0$ of system \eqref{general eq eps} and try to reduce the system at hand. We obtain in this singular limit the system
\begin{align}\label{general eq limit}
    \begin{split}
        0 &= F(u^0, v^0, 0), \\
             \partial_t v^{0} &= G(u^0, v^0, 0).
    \end{split}
\end{align}
This equation is a differential-algebraic differential equation, which is defined on the critical set
\begin{align*}
    C_0:= \{ (u^0,v^0)\in X\times Y: F(u^0,v^0,0)=0 \}.
\end{align*}
This limit system ignores the fast parts of the original system and only focuses on the dynamics of the slow variable and via this complexity reduction is therefore in general easier to study. For ODEs, the theory is quite well-understood and we briefly review it below. Yet, one major problem with this procedure for PDEs is that most differential operators that are defined on infinite dimensional spaces are not bounded, so taking the limit $\eps \to 0$ on the right-hand side of \eqref{general eq eps} may simply not be justified. The simplest example is the Laplace operator that is often used to model systems involving diffusion processes. Therefore, in the case of PDEs one has to be careful with obtaining the limit system, since taking the limit can potentially be much more singular than the ODE analog. \\

Next, we present a brief overview of the literature on fast-slow dynamical systems with a focus on the results in infinite dimensions. The theory for bounded operators, which leads to fast-slow ODEs dates back to the work by Tikhonov \cite{tikhonov1952systems} and Fenichel~\cite{fenichel1971persistence,fenichel1979geometric} and is known in the literature as geometric singular perturbation theory (GSPT). Suppose $S_0\subseteq C_0$ is a compact and normally hyperbolic submanifold, i.e., it is a sufficiently smooth manifold and the eigenvalues of the operator $A$ in the fast direction have nonzero real values for all $p\in S_0$.  Then the so called Fenichel-Tikhonov theory guarantees the existence of locally invariant slow manifolds $S^\eps$. For more details and applications of this theory we refer to the books by Jones \cite{jones1995geometric} and Kuehn \cite{kuehn2015multiple}. Yet, for infinite-dimensional systems new methods are required since applying the classical theory to PDEs does not hold and leads to problems as shown in \cite{hummel2022slow}.

One idea to overcome these problems is to adapt the theory of inertial manifolds \cite{chow1988invariant,FOIAS1988309}. The main goal of this theory is to construct a lower/finite-dimensional attracting invariant manifold. However, this approach requires a global dissipative structure of the PDE and compact embeddings to construct these lower-dimensional invariant manifolds. In addition, the inertial manifold theory relies on a spectral-gap type condition that allows the splitting of the underlying space into a finite-dimensional subspace and its complement.  

The next approach which is more along the lines of the classical fast-slow theory for ODEs is due to Bates et al. \cite{bates1998existence,bates2000invariant,bates2008approximately}. Here, for the case of semiflows the authors prove a local existence and persistence result of the so called slow manifold. In particular, their theory includes the case of partially dissipative systems, where the dynamics for the fast variable $u$ can be a PDE while the slow variable dynamics is an ODE. However, for fast-slow systems where both components are PDEs, this idea cannot be applied as was shown in \cite{hummel2022slow}. A different approach focusing on the construction of a slow manifold for differential inclusions can be found in \cite{gudovich2001tikhonov,Andreini2000result}.

In recent years progress was made for several classes of fast-slow PDEs. In \cite{hummel2022slow} a general theory is developed for the case of $F(u,v,\eps) = A u + f(u,v)$ and $G(u,v,\eps)= B u + g(u,v)$ with $A, B$ being closed linear (differential) operators and $f,g$ nonlinear functions. By introducing a modified notion of normal hyperbolicity the authors can prove, with few assumptions, the convergence of solutions of the system \eqref{general eq eps} to solutions of \eqref{general eq limit} in a suitable Banach space. Under slightly more stringent assumptions, one of which is a spectral-gap type condition for the operator $B$, the authors can also show the existence and persistence of slow manifolds. The assumption of normal hyperbolicity is key as was shown in \cite{engel2020blow} and some applications of the general theory can be found in \cite{engel2021connecting,hummel2022slow,Engeletal1}. For the class of fast-slow PDEs of the form $F(u,v,\eps) = \eps A u + f(u,v)$ and $G(u,v,\eps)= B u + g(u,v)$ the generalization of the classical results by Tikhonov and Fenichel was proven in \cite{kuehn2023fast}. In this work notion of normal hyperbolicity is adapted to account for the fact that the fast dynamics in this case are driven by the nonlinearity $f$. But again, a spectral-gap type condition is required to prove the existence and persistence of a slow invariant manifold. \\
%
In dynamical systems theory, understanding nonlinear systems is generally hard. However, the linear case is often easier to understand and it is usually a useful first step to gain deeper insights into the behavior of a system. In this work we study the linear case of the fast-slow systems analyzed in \cite{kuehn2023fast}, so called fast-reaction systems. Here we want to provide an explicit example, where all objects can be identified via formulas in contrast to the far more abstract setting in~\cite{kuehn2023fast}, which turns out to be quite instructive to learn about the problems arising in the geometric analysis of multiscale fast-reaction systems. These fast-reaction equations are an important subclass of fast-slow systems and occur in combination with diffusion operators for example in biology, chemistry or population dynamics. The mathematical study of fast-reaction systems dates back at least to the work by Hilhorst et al. \cite{hilhorst1996fast}. Over the last years there have been many results concerning the existence and uniqueness of solutions and the convergence to a limit system from a global functional-analytic perspective. A review on this topic can be found in \cite{iida2018review}. \\

Now, for the linear case of an infinite dimensional fast-reaction system we assume $f,g$ to be linear functions. Then, the aim of this work is to construct explicit solutions via Fourier methods and to show the convergence of solutions towards solutions of the limit system as the time scale parameter $\eps$ goes to zero. To do this we use typical Gronwall inequality estimates. Moreover, the slow manifold is \emph{explicitly} constructed, without requiring a spectral gap condition on the operator, and it converges to the critical manifold on the space of Fourier coefficients. This shows clearly that the nonlinear terms generate the difficulty regarding the mixing of different scales in the problem, which one would also conjecture intuitively from a physical perspective.

\section{Analysis of a linear fast-reaction system}

In the following, we present in detail the functional setting of the linear fast-reaction system at hand. 
As state spaces we choose $X:= Y: = L^2(\R^n)$ and for the linear operators we set $A:= \Delta -\mu I: H^2(\R^n)\to L^2(\R^n)$ and $B:= \Delta -\nu I: H^2(\R^n)\to L^2(\R^n)$, where $I$ denotes the identity on the Sobolev space $H^2(\R^n)$ and $\mu, \nu \in \R$ are additional parameters. 
As initial conditions we take $u_0$ and $v_0$, which are assumed to be in $H^2(\R^n)$.
Hence, we obtain a fully linear fast-reaction system the form
\begin{align} \label{linear system}
    \begin{split}
        \eps \partial_t u^{\eps}(t) &= \eps ( \Delta - \mu I) u^{\eps}(t) + \alpha u^{\eps}(t)+ \beta v^{\eps}(t), \\
             \partial_t v^{\eps}(t) &= ( \Delta - \nu I) v^{\eps}(t) + \gamma u^{\eps}(t) + \delta v^{\eps}(t),\\
                   u^\eps(0) &= u_0 \in H^2(\R^n), \quad v^\eps(0) = v_0 \in H^2(\R^n),
    \end{split}
\end{align}
for $t\geq 0$, where the non zero parameters $\beta,\gamma,\delta \in \R \backslash \{ 0 \}$ and $\alpha<0$ characterise the in this case linear functions $f$ and $g$. 
The associated limit system for $\eps\to 0$ is a differential-algebraic equation of the form
\begin{align} \label{linear limit system}
    \begin{split}
        0 &= \alpha u^0(t)+ \beta v^0(t), \\
        \partial_t v^0(t) &= (\Delta - \nu I) v^{0}(t) + \gamma u^{0}(t) + \delta v^{0}(t),\\
        u^0(0) &= - \alpha^{-1} \beta v^0(0) \in H^2(\R^n), \quad v^0(0)=v_0 \in H^2(\R^n)
    \end{split}
\end{align}
defined on a linear critical set 
\begin{align*}
    S_0 := \{ (u,v) \in H^2(\R^n) \times H^2(\R^n) : \alpha u + \beta v = 0\} \subset L^2(\R^n) \times L^2(\R^n).
\end{align*}

\subsection{Explicit solutions of the linear system}
Since both systems \eqref{linear system} and \eqref{linear limit system} are linear in $u^\eps$ and $v^\eps$, the existence and uniqueness of solutions follows in a straightforward way and we refer to \cite{evansPDE} for more details on the existence theory for linear PDEs. 
However, using Fourier analysis we can easily write down the explicit solutions of the systems. We start with the simpler system, the limit system.

\begin{lemma}
    The explicit solution to the limit system \eqref{linear limit system} has the form
    \begin{align*}
       \begin{pmatrix} 
       u^0(t)  \\ v^0(t) \end{pmatrix} = \begin{pmatrix}
         -\alpha^{-1}\beta v^0(t) \\ \e^{\kappa t} K_n(t) * v_0 \end{pmatrix}, 
    \end{align*}
    where $\kappa:= -\nu -\alpha^{-1}\beta\gamma +\delta $.
    Here, $K_n(t):= ( 4 \pi t)^{-\frac{n}{2}} \e^{ \frac{-|k|^2}{4t} }$ denotes the $n$-dimensional heat kernel, where
    $*$ is the usual convolution of measurable functions.
\end{lemma}
\begin{proof}
    Since $\alpha \neq 0$ the first equation of \eqref{linear limit system} gives $u^0(t) = - \alpha^{-1} \beta v^0(t)$. 
   The Fourier transform to some integrable function $w=w(t,x)$ is defined as
    $$
    (\F w)(t):= \int_{\R^n} w(t,x) \e^{- 2 \pi i k x} dx. $$
    Applying this to the equation for $v^0$ yields
\begin{equation*}
    \partial_t (\F v^{0})(t) = (-4 \pi^2 |k|^2 - \nu -\alpha^{-1}\beta \gamma + \delta ) \F v^{0}(t) 
\end{equation*}
for all $k \in \R^n$, since $\F( \Delta v^0) = - 4 \pi^2 |k|^2 \F(v^0)$. 
So, we can infer that
\begin{equation*}
    (\F v^{0})(t) = \e^{(-4 \pi^2 |k|^2 + \kappa ) t} (\F v_0),  
\end{equation*}
where $\kappa:= -\nu -\alpha^{-1}\beta\gamma +\delta$ and
which leads upon an inverse Fourier transformation to the claimed form.
Recall that the space $H^2(\R^n)$ can be defined via Fourier transforms
$$
H^2(\R^n):=\{ f\in L^2(\R^n)\,:\, (1+|k|^2) \F f\in L^2(\R^n)\}.
$$
Hence applying Plancherel's identity in the upcoming estimates it is sufficient to show that $\| (1+|k|^2) \F f\|_{L^2(\R^n)}$ remains bounded.
Then, from the regularity of the heat kernel and the initial data $v_0$ it follows that $v^0(t),u^0(t) \in H^2(\R^n)$.

\end{proof}

\begin{lemma} \label{solution lemma of linear system}
For small enough $\eps>0$ the explicit solution to the complete system \eqref{linear system} is given by
\begin{align*}
    \begin{pmatrix}
u^\eps(t) \\ 
v^\eps(t)
\end{pmatrix} 
= 
\begin{pmatrix} 
    [ \rho_1^\eps L^\eps_+(t) + \rho_2^\eps L^\eps_-(t) ] * u_0 + 
    [ \rho_3^\eps L^\eps_+(t) + \rho_4^\eps L^\eps_-(t) ] * v_0 \\ 
    [ \rho_5^\eps L^\eps_+(t) + \rho_6^\eps L^\eps_-(t) ] * u_0 + 
    [ \rho_7^\eps L^\eps_+(t) + \rho_8^\eps L^\eps_-(t) ] * v_0
    \end{pmatrix},
\end{align*}
with 
$$L^\eps_\pm(t):=\e^{ ( \epsi \alpha  - \mu + \delta - \nu \pm \Omega^\eps ) \frac{t}{2}} K_n(t)$$
and where 
$$\Omega^\eps := \sqrt{(\delta  - \nu - \epsi \alpha + \mu)^2 + 4 \epsi \beta \gamma} \, \in \R \backslash \{0\}.$$
The coefficients $\rho^\eps_i \in \R$ for $ i \in \{1,...,8\} $ depend only on the constants $\alpha,\beta,\gamma,\delta,\mu,\nu$ and the parameter $\eps$.
\end{lemma}

\begin{proof}
First, applying the Fourier transform to the linear system (\ref{linear system}) yields for $t>0$
\begin{align} \label{linear complete system fourier}
    \begin{split}
        \eps \partial_t (\F u^{\eps})(t) &= (- \eps 4 \pi^2 |k|^2 - \mu \eps +  \alpha) (\F u^{\eps} )(t)  +  \beta (\F v^{\eps} )(t), \\
            \partial_t (\F v^{\eps} )(t) &= \gamma(\F u^{\eps} )(t)  + (\delta - 4 \pi^2 |k|^2- \nu) (\F v^{\eps} )(t) \\
                   (\F u^\eps)(0) &= (\F u_0) \in H^2(\R^n), \quad (\F v^\eps)(0) = (\F{v_0}) \in H^2(\R^n)
    \end{split} 
\end{align}
for all $k\in\R^n$. 
To simplify the notation write the above system as a two-dimensional ODE for each Fourier coefficient $k$, where we set
\begin{equation*}
    w_k^\eps(t):=
    \begin{pmatrix}
    (\F u^\eps)(t) \\ ( \F v^\eps)(t)
    \end{pmatrix}  \in \R^2.
\end{equation*} 
Then, multiplying the first equation of system (\ref{linear complete system fourier}) by $\epsi$, we can reformulate the system as 
$$\partial_t w_k^\eps(t) = M_k^\eps w_k^\eps(t) $$ 
where
\begin{equation*}
M_k^\eps := 
\begin{pmatrix}
\epsi \alpha - \mu - 4 \pi^2 |k|^2 & \epsi \beta  \\
\gamma & \delta -\nu - 4 \pi^2 |k|^2  \\
\end{pmatrix} \in \R^{2 \times 2}.
\end{equation*}
The eigenvalues of $M^\eps_k$ are given by
$$ 
\lambda^\eps_{\pm}(k) := \frac{1}{2}\left(\epsi \alpha - \mu + \delta -\nu - 8\pi^2 |k|^2  \mp \Omega^\eps \right) \in \R.
$$
with
$$ 
\Omega^\eps := \sqrt{(\delta  - \nu - \epsi \alpha + \mu)^2 + 4 \epsi \beta \gamma}.
$$
Regardless of the values of the parameters, one can always ensure that 
$$(\delta  -\nu   - \epsi \alpha + \mu )^2> - 4\epsi \beta \gamma$$
for all sufficiently small $\eps$. 
Hence, we can assume that $\eps$ is small enough such that $\Omega^\eps \in \R_+ $ and $\lambda^\eps_{\pm} \in \R$.
Now, if $|k|^2 = \frac{ \epsi \alpha -  \mu + \delta  -\nu \mp \Omega^\eps}{8 \pi^2 } =: \rho^\eps_\pm$ we have $\lambda^\eps_{\pm}(k)=0$. But since the sets 
$$ 
Q_\pm:=\{ k \in \R^n : |k|^2 = \rho^\eps_\mp \}
$$
are $(n-1)$-dimensional spheres of radius $\sqrt{\rho_\pm^\eps}$ if $\rho_i^\eps>0$, $Q_\pm$ are null sets with respect to the Lebesgue measure in $\R^n$. 
So in the following we can assume that 
we have non-zero eigenvalues for almost every $k\in \R$.\\
Therefore, we can derive a formula for the solution of system (\ref{linear complete system fourier}) using matrix diagonalization of $M_k^\eps$, i.e. we can write for almost every $k$
$$ 
M^\eps_k = S^\eps J^\eps_k ({S^\eps})^{-1} 
$$ 
with
\begin{align*} 
S^\eps&:=
\frac{1}{2\gamma}\begin{pmatrix}
\epsi \alpha -\mu-\delta + \nu  -\Omega^\eps & \epsi \alpha -\mu-\delta + \nu  +\Omega^\eps \\ 2\gamma & 2\gamma
\end{pmatrix}  \in \R^{2 \times 2}
\intertext{ and }
J^\eps_k&:=
\begin{pmatrix} \lambda^\eps_{-}(k) & 0\\0 & \lambda^\eps_{+}(k)
\end{pmatrix} \in \R^{2 \times 2}. 
\end{align*} 
Using the diagonalization the solution of system (\ref{linear complete system fourier}) can be written as 
$$ 
\begin{pmatrix}
(\F u^\eps)(t) \\
(\F v^\eps)(t)
\end{pmatrix} = S^\eps  \e^{J_k^\eps t} (S^\eps)^{-1} \begin{pmatrix}
(\F u_0) \\ (\F v_0)
\end{pmatrix}  .
$$ 
After performing the matrix multiplications we get 
\begin{align*}
(\F u^\eps)(t) = \frac{1}{{2 \Omega^\eps}} \bigg(&
{( \epsi \alpha -\mu- \delta +\nu + \Omega^\eps )}  \e^{\lambda^\eps_-(k) t} \\
&- ({ \epsi \alpha -\mu- \delta +\nu - \Omega^\eps}) \e^{\lambda^\eps_+(k) t}  
\bigg) (\F u_0) \\ 
- \frac{\epsi \beta }{  \Omega^\eps} \bigg(& 
\e^{\lambda_-^\eps(k) t} -  \e^{\lambda_+^\eps(k) t} 
\bigg) (\F v_0) 
\intertext{and}
(\F v^\eps)(t) = \frac{\gamma}{ \Omega^\eps }  \bigg(&  \e^{\lambda^\eps_-(k) t} - \e^{\lambda^\eps_+(k) t} \bigg) (\F u_0) \\ 
 + \frac{1}{{2 \Omega^\eps}} \bigg( & { -( \epsi \alpha -\mu - \delta +\nu + \Omega^\eps )} \e^{\lambda^\eps_-(k) t } \\
 &- {( \epsi \alpha -\mu- \delta +\nu- \Omega^\eps)} \e^{\lambda^\eps_+(k) t} \bigg) (\F v_0). 
\end{align*}
Note that except for the initial conditions $\F u_0$ and $\F v_0$, the only term in the above expressions that depends on $k$ is the exponential term $\e^{\lambda^\eps_{\pm}(k)t }$. Thus, all that remains is to perform the inverse Fourier transform of this function, which evaluates to be $L^\eps_\pm(t)$. By the linearity of the Fourier transform, we obtain the claimed solution for system \eqref{linear system}. Since the Fourier transform involves integration over $\R^n$, we can neglect that the solution formula for the Fourier transform is valid only for almost every $k \in \R^n$. \\
Moreover, we observe that $(1+|k|^2)\rho_i^\eps e^{\lambda^\eps_\pm(t)} \F (u_0)\in L^2(\R^n)$ for $t>0$ and, respectively, the same holds for $v_0$.
Hence the solution satisfies $u^\eps(t), v^\eps(t)\in H^2(\R^n)$.
\end{proof}
\begin{remark}
We note that all coefficients $\rho_i^\eps$ introduced in the previous proof are well defined in the limit as $\eps \to 0$.
Indeed, we have
$\lim_{\eps \to 0}(\Omega^\eps)^{-1}= 0 $ $\lim_{\eps \to 0}\epsi (\Omega^\eps)^{-1}= \alpha^{-1}$, which yields the desired bounds for the coefficients.
\end{remark}
\subsection{Convergence of solutions}
As in \cite{hummel2022slow} and \cite{kuehn2023fast} in order to solve the limit system \eqref{linear limit system} we have introduced a function 
$$h^0: H^2(\R^n) \to H^2(\R^n),\, y \mapsto h^0(y)$$ 
such that the critical manifold $S_0$ can be written as
\begin{align*}
    S_0=\{ (h^0(y),y): y\in H^2(\R^n)\}.
\end{align*} 
In the case of system \eqref{linear limit system} the function $h^0$ is linear and given by $h^0(v^0) := \alpha ^{-1} \beta v^0 = u^0$.\\
In addition, we observe that the two systems
are autonomous and hence we can write the solutions $(u^\eps,v^\eps)$ and $(u^0,v^0)$ as semi-flows. 
That is, there exist continuous mappings
\begin{align*}
    T_\eps : [0,T] \times H^2(\R^n) \times H^2(\R^n) \to H^2(\R^n)\times H^2(\R^n) \text{ and } T_0:[0,T] \times S_0 \to S_0
\end{align*}
for any $T>0$ such that the solutions can be written as
\begin{align*}
    \begin{pmatrix}
    u^\eps(t) \\ v^\eps(t)
    \end{pmatrix} = T_\eps (t) \begin{pmatrix}
    u_0 \\ v_0
    \end{pmatrix} \text{ and } \begin{pmatrix}
    u^0(t) \\ v^0(t)
    \end{pmatrix} = T_0(t)\begin{pmatrix}
    h^0(v_0) \\ v_0
    \end{pmatrix}.
\end{align*}

The following theorem provides a convergence result of the semi-flows of the linear system \eqref{linear system} towards solutions of the limit system \eqref{linear limit system} as $\eps \to 0$. 

\begin{theorem} \label{main main theorem}
    Let $\eps_0>0$ be small enough such that $\epsi_0 \alpha-\mu<0$ and $T>0$. Then, for all $\eps\in (0,\eps_0]$ there exist a positive increasing function $C_{\nu,\kappa}(t) $ depending on the parameters $\nu$ and $\kappa:= -\nu -\alpha^{-1}\beta\gamma +\delta$ such that
    \begin{align*}
    &\left\| 
    T_{\eps}(t) \begin{pmatrix}
    u_0  \\ v_0  \end{pmatrix} 
    - T_{0}(t) \begin{pmatrix}
    h^0(v_0)  \\ v_0 \end{pmatrix} 
    \right\|_{ H^2(\R^n) \times H^2(\R^n) } 
    \leq \\
    &\quad \leq \eps C_{\nu,\kappa}(t) \bigg( \|u_0 - h^0(v_0)\|_{H^2(\R^n)} + \|v_0\|_{H^2(\R^n)}\bigg), 
    \end{align*}
for all $t\in (0,T)$.
In the case $\nu>0$ and $\kappa<0$ we can set $C_{\nu,\kappa}(t)=C>0$.
\end{theorem}

\subsection*{Proof of Theorem \ref{main main theorem}}

The proof of Theorem \ref{main main theorem} is split into several steps.
The overall idea is to introduce two auxiliary systems that in some sense interpolate between the equation of the fast variable and its algebraic limit.  Then, by using a variation of constants method in combination with the Fourier transform we obtain estimates for the Fourier coefficients and thus for the solutions themselves. 

In the first step we introduce the two approximate systems in the fast component $u^\eps$.
That is, we have $u^{\eps,0}$ as the solution to
\begin{align}\label{linear extended system}
    \eps \partial_t u^{\eps,0}(t) &= \eps (\Delta - \mu I)  u^{\eps,0}(t) + \alpha u^{\eps,0}(t) + \beta v^0(t) + \eps ( \partial_t - \Delta + \mu I) h^0 (v^0(t))
\end{align}
and $\tilde u^\eps$ solving
\begin{align}\label{linear fast tilde equation}
     \eps \partial_t \tilde u^{\eps}(t) &= \eps (\Delta -  \mu I) \tilde   u^{\eps}(t) + \alpha \tilde  u^{\eps}(t) + \beta v^0(t),
\end{align}
where in both cases the equation in the slow variable $v$ is given by
\begin{align} \label{v0}
     \partial_t v^0(t) &= (\Delta - \nu I) v^{0}(t) - \gamma \beta \alpha^{-1} v^0(t) + \delta v^{0}(t)
\end{align}
and the initial data satisfy
\begin{align*}
    u^{\eps,0}(0)=u_0,\quad \tilde  u^\eps (0)=u_0,\quad v^0(0)= v_0.
\end{align*}
The motivation behind introducing this approximate system $u^{\eps,0}$  is that the solution gets close to the solution of the limit system \eqref{linear limit system}.
A formal way to see this is to consider the different powers of $\eps$ in the equation.
In $\mathcal{O}(1)$ we have
$$ 
    \alpha u^{\eps,0}(t) + \beta v^0(t) =0 
$$
and in $\mathcal{O}(\eps)$ we get 
$$
    \partial_t (u^{\eps,0} - h^0(v^0(t))) = (\Delta - \mu I) ( u^{\eps,0} - h^0(v^0(t))) 
$$
which can be viewed as $\alpha u^{\eps,0} + \beta v^0 = 0$, since for small $\eps$ it holds that $ u^{\eps,0} \approx h^0(v^0)$. Thus, in both orders of $\eps$ we obtain an equation corresponding to an equation characterizing the critical manifold of system \eqref{linear limit system}.

The well-posedness of the auxiliary systems follows from using the linearity of the systems and by noting that the equation in $v$ is independent of the fast variable $u$. 
Therefore, one can first solve the $v$-equation and then the equation in $u^{\eps,0}$ and $\tilde u^\eps$ respectively. \\

The next step of the proof is to estimate the fast variable. 
By applying a triangle inequality we have
\begin{align*}
   \| u^\eps -h^0(v^0)  \|_{H^2(\R^n)} \!\leq   \| u^\eps - \tilde {u}^{\eps} \|_{H^2(\R^n)} \!+   \| \tilde {u}^\eps  - u^{\eps,0} \|_{H^2(\R^n)} \!+ \| u^{\eps,0} - h^0(v^0) \|_{H^2(\R^n)}.
\end{align*}
In the following proposition we give further estimates for each term. 

\begin{proposition} \label{auxillary lemma}
Let $\alpha<0$ and assume that $\eps>0$ is small enough such that $\epsi \alpha < \mu$. Then if $u^\eps$ is the solution of equation \eqref{linear system} and $\tilde {u}^\eps $ the solution to equation \eqref{linear fast tilde equation} it holds for $t \geq 0$

\begin{equation*}
    \| u^\eps(t) - \tilde {u}^\eps(t) \|_{H^2(\R^n)} 
    \leq  C  \sup_{s \in [0,t]} \|v^\eps(s)- v^0(s) \|_{H^2(\R^n)} , 
\end{equation*}
for some constant $C>0$.
\end{proposition}
\begin{proof}
Applying the Fourier transform to the first equation in \eqref{linear system} and multiplying by $\epsi$ gives 
$$ 
    \partial_t (\F u^\eps)(t) = (- 4 \pi^2 |k|^2 - \mu + \epsi \alpha) (\F u^\eps)(t) + \epsi \beta (\F v^\eps)(t)
$$

for $k \in \R^n$. By variation of constants we have 
$$ 
    (\F u^\eps)(t) = \e^{(-4 \pi^2 |k|^2 - \mu + \epsi \alpha)t} (\F u_0) + \epsi \beta \int_{0}^t \e^{(-4 \pi^2 |k|^2 - \mu + \epsi \alpha)(t-s)} (\F v^\eps)(s) \textnormal{d} s.
$$ 
On the other hand,
$$ 
    (\F \tilde {u}^\eps)(t) = \e^{(-4 \pi^2 |k|^2 - \mu + \epsi \alpha)t} (\F u_0) + \epsi \beta \int_{0}^t \e^{(-4 \pi^2 |k|^2 - \mu + \epsi \alpha)(t-s)} (\F v^0)(s) \textnormal{d} s.
$$ 
So by Plancherel's theorem and a generalized H\"older's inequality we obtain
\begin{align*}
    \|  u^\eps(t) -  \tilde {u}^\eps(t) \|_{H^2(\R^n)} & =    \| (1+|k|^2)\big((\F u^\eps)(t)- (\F \tilde {u}^\eps)(t)\big) \|_{L^2(\R^n)} \\  
    &\leq |\beta| \epsi \int_{0}^t \bigg(\| \e^{(-4 \pi^2 |k|^2 - \mu+ \epsi \alpha)(t-s)} \|_{L^\infty(\R^n)} \\
    & \qquad \qquad \qquad \times \| (1+|k|^2)\big((\F v^\eps)(s) - (\F v^0)(s) \big)\|_{L^2(\R^n)}\bigg) \textnormal{d} s \\
    &\leq |\beta| \sup_{s \in [0,t]} \|  (1+|k|^2)\big((\F v^\eps)(s) - (\F v^0)(s) \big) \|_{L^2(\R^n)} \\
    &\quad \times \epsi \int_{0}^t \e^{ (- \mu + \epsi \alpha)(t-s)} \textnormal{d} s \\
    & =  |\beta| \sup_{s \in [0,t]}  \| (1+|k|^2)\big((\F v^\eps)(s) - (\F v^0)(s) \big) \|_{L^2(\R^n)} \\ 
    &\quad \times \frac{1- \e^{(-\mu + \epsi \alpha)t} }{\mu -\epsi \alpha} \\
    & \leq \frac{ | \beta | }{\mu -\epsi \alpha} \sup_{s \in [0,t]}  \| v^\eps (s) - v^0(s) \|_{H^2(\R^n)}\\
    & \leq C \sup_{s \in [0,t]}  \| v^\eps (s) - v^0(s) \|_{H^2(\R^n)}.
\end{align*}
\end{proof}

\begin{proposition}
    Let $\alpha<0$ and assume that $\eps>0$ is small enough such that $- \mu + \epsi \alpha < \kappa :=  - \nu -\beta\alpha^{-1} \gamma + \delta$.
    Then for $t \geq 0$ there is a $C>0$ such that
    \begin{align*}
        \| u^{\eps,0}(t)- \tilde  u^\eps(t)\|_{H^2(\R^n)}\leq \eps C \textnormal{e}^{\kappa t} \|v_0\|_{H^2(\R^n)}.
    \end{align*}
\end{proposition}

\begin{proof}
Since by equation \eqref{v0} $(\partial_t - \Delta) v^0(t) =(- \nu -\beta\alpha^{-1} \gamma + \delta )v^0(t)= \kappa v^0(t) $, equation \eqref{linear extended system} can be written as 
\begin{align*} 
    \eps \partial_t u^{\eps,0}(t) = \eps (\Delta - \mu I) u^{\eps,0}(t) + \alpha u^{\eps,0}(t) +  \left( \beta  -\eps \beta\alpha^{-1} (  \mu +\kappa ) \right)  v^0(t).
\end{align*}
Applying the Fourier transform and multiplying with $\epsi$ gives
\begin{align*}
    \partial_t (\F u^{\eps,0})(t) &= (- 4\pi^2 |k|^2 - \mu + \epsi \alpha) (\F u^{\eps,0})(t) \\
    &\quad +   \left( \epsi \beta   -\beta\alpha^{-1} (  \mu +\kappa ) \right) \e^{(-4 \pi^2 |k|^2  +\kappa) t} (\F v_0) 
\end{align*}
for all $k \in \R^n$. By variation of constants we get 
\begin{align*}
    (\F u^{\eps,0})(t) &= \e^{ (-4 \pi^2 |k|^2 - \mu + \epsi \alpha) t} (\F u_0) \\ &+ 
    \left( \beta \epsi -\alpha^{-1}\beta ( \mu +\kappa ) \right) \int_0^t  \e^{(-4 \pi^2 |k|^2 - \mu + \epsi \alpha)(t-s)} \e^{ (- 4\pi^2 |k|^2 +\kappa) s} \textnormal{d} s (\F v_0).
\end{align*}
By performing the integration in the latter expression we obtain
\begin{align*}
    \int_0^t  \e^{(-4 \pi^2 |k|^2 - \mu + \epsi \alpha)(t-s)} \e^{ (- 4\pi^2 |k|^2 + \kappa) s} \textnormal{d} s 
    &= \e^{ (-4 \pi^2 |k|^2 - \mu + \epsi \alpha) t} \int_0^t \e^{ (  - \epsi \alpha + \mu+\kappa ) s} \textnormal{d} s  \\ 
    &= \frac{\e^{ (-4 \pi^2 |k|^2 - \mu + \epsi \alpha) t}}{ - \epsi \alpha + \mu +\kappa} \left(\e^{ (- \epsi \alpha + \mu+\kappa ) t} -1 \right)  \\
    &= \frac{\e^{ -4 \pi^2 |k|^2 t}}{ - \epsi \alpha + \mu+\kappa }  \left(\e^{ \kappa t } -\e^{ ( - \mu + \epsi \alpha) t}\right). 
\end{align*}
Therefore, after an inverse Fourier transformation we obtain
\begin{align*}
    u^{\eps,0}(t)\! &= \e^{( -\mu + \epsi \alpha) t} K_n(t)\! *\! u_0 + \frac{ \beta - \eps  \alpha^{-1}\beta ( \mu +\kappa )}{ - \alpha + \eps \mu+ \eps\kappa } \left(\e^{ \kappa t} -\e^{ (- \mu + \epsi \alpha) t} \right) K_n(t)\! *\! v_0.
\end{align*} 
By similar calculations we can derive from equation \eqref{linear fast tilde equation} that
\begin{equation*}
    \tilde {u}^{\eps}(t)= \e^{( -\mu + \epsi \alpha) t} K_n(t) * u_0 + \frac{ \beta}{ - \alpha + \eps \mu+ \eps\kappa } \left(\e^{\kappa t} -\e^{ (- \mu + \epsi \alpha) t} \right) K_n(t) * v_0 .
\end{equation*}
Hence, using that we have $-\mu + \epsi \alpha < \kappa $ yields 
\begin{align*}
    \| u^{\eps,0}(t) - \tilde {u}^\eps(t) \|_{H^2(\R^n)} 
    &= \frac{\eps |-\alpha^{-1}\beta ( \mu +\kappa )|}{|- \alpha + \eps \mu+ \eps\kappa|} \left| \e^{\kappa t}-  \e^{ (- \mu + \epsi \alpha) t}  \right| \|v_0\|_{H^2(\R^n)} \\
    & \leq C \frac{\eps |-\alpha^{-1}\beta ( \mu +\kappa )|}{|- \alpha + \eps \mu+ \eps\kappa|} \e^{\kappa t} \|v_0\|_{H^2(\R^n)} \\
    & \leq \eps C \e^{\kappa t}\|v_0\|_{H^2(\R^n)}.
\end{align*}
for some positive constant C.
\end{proof}

\begin{proposition}
 Let $\eps > 0$ be small enough such that $-\mu + \epsi \alpha<0$ . Then there exists a constant  $C>0$ such that
\begin{equation*}
  \|u^{\eps,0}(t)-h^0(v^0(t))  \|_{H^2(\R^n)} 
    \leq C \e^{(-\mu + \epsi \alpha)t} \| u_0 - h^0(v_0) \|_{H^2(\R^n)}.
\end{equation*}
\end{proposition}

\begin{proof}
As in the previous proof we apply the variation of constant method along with the Fourier transform. 
Hence, we rewrite the solution $u^{\eps,0}$ as
$$
u^{\eps,0}(t) = \e^{( -\mu + \epsi \alpha) t} K_n(t) * u_0 -\alpha^{-1}\beta \left(\e^{\kappa t} -\e^{ (- \mu + \epsi \alpha) t} \right) K_n(t) * v_0
$$
and obtain
\begin{align*}
    &\|u^{\eps,0}(t)-h^0(v^0(t))  \|_{H^2(\R^n)}=\\
    &= \| \e^{( -\mu + \epsi \alpha) t} (1+|k|^2) K_n(t) \!*\! u_0 -\alpha^{-1}\beta  \left(\e^{\kappa t} -  \e^{ (- \mu + \epsi \alpha) t} \right)  (1+|k|^2)K_n(t)\! *\! v_0\\
    &\quad \quad +\alpha^{-1}\beta \e^{ \kappa t}  (1+|k|^2) K_n(t) * v_0 \|_{L^2(\R^n)} \\
    &= \| \e^{( -\mu + \epsi \alpha) t} (1+|k|^2) K_n(t) * \left[  u_0 +\alpha^{-1}\beta v_0  \right] \|_{L^2(\R^n)} \\
    &\leq C \e^{(-\mu + \epsi \alpha)t} \| u_0 - h^0(v_0) \|_{H^2(\R^n)}.
\end{align*}
\end{proof}    

Combining the three estimates from the previous propositions with $\eps$ small enough to fulfil the required conditions and $\alpha<0$ yields for a $C>0$
\begin{align}\label{u estimate}\begin{split}
   \| u^\eps(t) -h^0(v^0(t))  \|_{H^2(\R^n)}&\leq C \sup_{s \in [0,t]} \|v^\eps(s)- v^0(s) \|_{H^2(\R^n)} +  \eps C \e^{\kappa t}\|v_0\|_{H^2(\R^n)}\\
   &\quad +C \e^{(-\mu + \epsi \alpha)t} \| u_0 - h^0(v_0) \|_{H^2(\R^n)}.\end{split}
\end{align}

The next step in the proof is to estimate $\|v^\eps(t) - v^0(t) \|_{H^2(\R^n)}$. 
By variations of constants we obtain
$$
(\F v^\eps)(t) = \e^{(-4 \pi^2 |k|^2 - \nu) t} (\F v_0) + \int_0^t \e^{(-4 \pi^2 |k|^2 - \nu) (t-s)} (\gamma (\F u^\eps)(s) + \delta (\F v^\eps)(s) ) \textnormal{d} s
$$
and 
$$
(\F v^0)(t) = \e^{ (-4 \pi^2 |k|^2 - \nu ) t} (\F v_0) + \int_0^t \e^{(-4 \pi^2 |k|^2 - \nu ) (t-s)} (\delta-\gamma \alpha^{-1}\beta )(\F v^0)(s)  \textnormal{d} s.
$$
Hence, 
\begin{align*}
\| v^\eps(t) - v^0(t) \|_{H^2(\R^n)}&\leq \int_0^t \bigg(\|\e^{(-4 \pi^2 |k|^2- \nu)(t-s)} \|_{L^\infty(\R^n)}|\delta| \|v^\eps(s) -v^0(s)\|_{H^2(\R^n)})\\ 
&+ \|\e^{(-4 \pi^2 |k|^2- \nu)(t-s)} \|_{L^\infty(\R^n)} (|\gamma| \|u^\eps(s) - h^0(v^0(s))\|_{H^2(\R^n)}\bigg) \textnormal{d} s.
\end{align*}
Inserting the estimate for the fast variable \eqref{u estimate} yields
\begin{align*}
&\| v^\eps(t) - v^0(t) \|_{H^2(\R^n)} \leq \\
&\leq \int_0^t \e^{- \nu(t-s)} \bigg( C {  
\e^{(-\mu + \epsi \alpha)s} \|u_0 - h^0(v_0)\|_{H^2(\R^n)}
+ \eps C  \e^{\kappa s} \|v_0\|_{H^2(\R^n)}
 }  \\
& \qquad +  |\delta|\|v^\eps(s) -v^0(s)\|_{H^2(\R^n)} + C \sup_{r\in [0,s]}\|v^\eps(r) -v^0(r)\|_{H^2(\R^n)}  \bigg) \textnormal{d} s.
\end{align*}
Since the right-hand side of the inequality is increasing in time, we obtain
\begin{align*}
& \sup_{s\in [0,t]}\| v^\eps(s) - v^0(s) \|_{H^2(\R^n)} \\
&\leq C\int_0^t \e^{- \nu(t-s)} \big( {  
\e^{(-\mu +\epsi \alpha)s} \|u_0 - h^0(v_0)\|_{H^2(\R^n)}
+ \eps   \e^{ \kappa s} \|v_0\|_{H^2(\R^n)}
 }\big)  \textnormal{d} s\\
& \quad + C\int_0^t  \e^{- \nu(t-s)} \sup_{r\in [0,s]}\|v^\eps(r) -v^0(r)\|_{H^2(\R^n)}   \textnormal{d} s.
\end{align*}
This expression can be further simplified to 
\begin{align*}
    & \sup_{s\in [0,t]}\| v^\eps(s) - v^0(s) \|_{H^2(\R^n)} \leq\\
&\leq C \e^{-\nu t} \bigg(   \|u_0 - h^0(v_0)\|_{H^2(\R^n)}  \frac{\e^{(\epsi \alpha - \mu + \nu)t} -1}{ \epsi \alpha - \mu +  \nu}
+  \eps \|v_0\|_{H^2(\R^n)} \frac{\e^{( -\alpha^{-1}\beta \gamma + \delta)  t}-1} { -\alpha^{-1}\beta \gamma + \delta }\bigg)\\
& \quad +C \e^{- \nu t}\int_0^t  \e^{ \nu s}\sup_{r\in [0,s]}\|v^\eps(r) -v^0(r)\|_{H^2(\R^n)}   \textnormal{d} s.
\end{align*}
This estimate is of the form 
$$z(t)\leq A(t)+ \int_0^t B(s)z(s) \textnormal{d} s.$$
Hence, applying Gronwall's inequality yields 
$$z(t)\leq A(t) +\int_0^t A(s)B(s) \exp\bigg(\int_s^t B(r)dr\bigg) \textnormal{d} s.$$
We observe that for $0\leq s\leq t$
\begin{align*}
     \exp\bigg(C\int_s^t  \e^{-\nu(t-r)} dr\bigg)= \exp C\nu^{-1}  \big( 1 - \e^{-\nu(t-s) } \big)\leq \exp C\nu^{-1}  \big( 1 - \e^{-\nu t } \big) \leq  C_\nu(t),
\end{align*}
where $C_\nu(t)=C_\nu\geq 0$ is a constant if $\nu >0$ and otherwise $C_\nu(t)$ is a positive increasing function.  
Therefore, we can estimate
\begin{align*}
      &\int_0^t A(s)B(s) \exp\bigg(\int_s^t B(r)dr\bigg) \textnormal{d} s \leq\\
      & \leq C_\nu(t) \int_0^t   \|u_0 - h^0(v_0)\|_{H^2(\R^n)}  \frac{\e^{(\epsi \alpha - \mu )s} -\e^{-\nu s}}{ \epsi \alpha - \mu + \nu} \textnormal{d} s\\
&\quad +C_\nu(t)  \int_0^t \e^{-\nu s} \bigg(\eps \|v_0\|_{H^2(\R^n)} \frac{\e^{ (-\alpha^{-1}\beta \gamma + \delta )  s}-1} {-\alpha^{-1}\beta \gamma + \delta }\bigg) \textnormal{d} s\\
&\leq I_1+I_2.
\end{align*}
Evaluating each integral yields
\begin{align*}
    I_1 & \leq \frac{\eps C_\nu(t)}{ | \alpha - \eps\mu + \eps\nu|} \|u_0 - h^0(v_0)\|_{H^2(\R^n)} \bigg( \frac{1-\e^{(\epsi \alpha - \mu )t}}{\mu -\epsi \alpha } + \frac{\e^{-\nu t}-1}{\nu}\bigg)
    \intertext{and}
    I_2&\leq \frac{\eps C_\nu(t)}{|-\alpha^{-1}\beta \gamma + \delta|} \|v_0\|_{H^2(\R^n)} \bigg( \frac{\e^{ \kappa  t}-1}{\kappa}+\frac{\e^{-\nu t}-1}{\nu}\bigg).
\end{align*}
Hence we obtain
\begin{align*}
      & \sup_{s\in [0,t]}\| v^\eps(s) - v^0(s) \|_{H^2(\R^n)} \leq\\
&\leq \eps C_\nu(t)  \|u_0 - h^0(v_0)\|_{H^2(\R^n)}\bigg( \frac{\e^{(\epsi \alpha - \mu + \nu)t} -\e^{-\nu t}}{\alpha - \eps\mu + \eps\nu}+\frac{1-\e^{(\epsi \alpha - \mu )t}}{\mu -\epsi \alpha } + \frac{\e^{-\nu t}-1}{\nu}\bigg)\\
&\quad +\eps C_\nu(t) \|v_0\|_{H^2(\R^n)} \bigg( \frac{\e^{\kappa  t}-\e^{-\nu t}} { -\alpha^{-1}\beta \gamma + \delta }+\frac{\e^{ \kappa  t}-1}{\kappa }+\frac{\e^{-\nu t}-1}{\nu}\bigg).
\end{align*}
Again depending on the sign of $\nu$ and $\kappa$ we can write this estimate as
\begin{align*}
    \sup_{s\in [0,t]}\| v^\eps(s) - v^0(s) \|_{H^2(\R^n)} \leq \eps C_{\nu,\kappa}(t) \big(  \|u_0 - h^0(v_0)\|_{H^2(\R^n)} +\|v_0\|_{H^2(\R^n)}\big),
\end{align*}
where $ C_{\nu,\kappa}(t)$ is a positive increasing function.
In the case that $nu>0$ and $\kappa<0$ we have $ C_{\nu,\kappa}(t)=C$.
Moreover, we observe that the right hand side is increasing in time and thus 
\begin{align*}
    \| v^\eps(t) - v^0(t) \|_{H^2(\R^n)} \leq \eps C_{\nu,\kappa}(t) \big(  \|u_0 - h^0(v_0)\|_{H^2(\R^n)} +\|v_0\|_{H^2(\R^n)}\big).
\end{align*}
Plugging this into the estimate for the fast variable \eqref{u estimate} yields
\begin{align*}
   & \| u^\eps(t) -h^0(v^0(t))  \|_{H^2(\R^n)}\leq \\
   &\leq  \eps C_{\nu,\kappa} \bigg(  \|u_0 - h^0(v_0)\|_{H^2(\R^n)} + \|v_0\|_{H^2(\R^n)}\bigg) + \e^{(\epsi \alpha-\mu)t}\|u_0 - h^0(v_0)\|_{H^2(\R^n)}.
\end{align*}
The final step in the proof is to combine the two estimates, resulting in 
 \begin{align*}
   & \left\| 
    T_{\eps}(t) \begin{pmatrix}
    u_0  \\ v_0  \end{pmatrix} 
    - T_{0}(t) \begin{pmatrix}
    h^0(v_0)  \\ v_0 \end{pmatrix} 
    \right\|_{ H^2(\R^n) \times H^2(\R^n) } \leq \\
   & \leq \eps C_{\nu,\kappa} \bigg( \big(1+\epsi\e^{(\epsi \alpha-\mu)t}\big ) \|u_0 - h^0(v_0)\|_{H^2(\R^n)} + \|v_0\|_{H^2(\R^n)}\bigg),
    \end{align*}
    for $t> 0$.
This completes the proof of Theorem \ref{main main theorem}.

\section{Dynamics of a linear fast-reaction system}

As mentioned in the introduction finite dimensional fast-slow systems tend to have, under certain hyperbolicity assumptions, so called slow manifolds.
These slow manifolds are a family $S_\eps$ of locally invariant manifolds that converge to the critical manifold as $\eps \to 0$ with respect to the Hausdorff-distance, while also preserving the regularity. Moreover, the flow on those manifolds converges to the flow of the limit system as  $\eps \to 0$. This means  that the existence of slow manifolds imply that the limit system is a reasonable good approximation of the whole system near the critical manifold. \\
It is therefore desirable to also have slow manifolds in the case of infinite dimensional systems. However, one should be reminded that infinite dimensional manifolds are much more difficult to handle. \\

In this section we show that slow manifolds can be obtained for the linear fast-reaction system \eqref{linear system} and provide an explicit formula. This means that the results from finite dimensional fast-slow systems generalize to the infinite dimensional linear fast-reaction systems. 

The first step of adapting the Fenichel-Tikhonov theory to this linear fast-reaction system is to construct the invariant slow manifold of the system. 
To bridge the gap between the finite-dimensional theory and its infinite-dimensional extension, we apply the Fourier transform to the linear system  \eqref{linear system} and construct the slow manifolds for each Fourier coefficient $k\in \R$.
Then, the following lemma holds.
\begin{lemma}
Let $\alpha<0$. Then, for all $\eps>0$ small enough satisfying $\epsi \alpha- \mu<0$ the slow manifolds of system \eqref{linear complete system fourier} are given by 
\begin{equation}
C^{\F}_\eps:=\{ (x,y) \in \R^2 : (\eps (\delta-\nu +\mu) - \alpha - \eps \Omega^\eps)x - 2 \beta y = 0 \},
\end{equation}
where $\Omega^\eps := \sqrt{(\delta-\nu - \epsi \alpha+\mu)^2 + 4 \epsi \beta \gamma} \in \R$.
\end{lemma}
\begin{proof}
Recall that after applying a Fourier transform to system \eqref{linear system} we can write the system as
\begin{align}\label{linear system fourier}
    \partial_t w_k^\eps(t)= M_k^\eps w_k^\eps(t),
\end{align}
where
\begin{align*}
    M_k^\eps := 
\begin{pmatrix}
\epsi \alpha - \mu - 4 \pi^2 |k|^2 & \epsi \beta  \\
\gamma & \delta -\nu - 4 \pi^2 |k|^2  \\
\end{pmatrix} \in \R^{2 \times 2}.
\end{align*}
The existence of slow manifolds follows directly from the theory for fast-slow ODEs, since in this case we have that the normal hyperbolicity condition is fulfilled since $\epsi \alpha- \mu<0$.
Moreover, we can explicitly construct the slow manifolds by computing the eigenspaces corresponding to the linear ODE system. 
Since system (\ref{linear system fourier}) is linear and finite-dimensional, the only locally invariant manifolds are submanifolds of the two eigenspaces $E^\eps_+$ and $E^\eps_-$ of $M^\eps_k$.
As we have seen in the proof of Lemma \ref{solution lemma of linear system}, for $\eps$ small enough there exist two non identical eigenspaces of $M_k^\eps$ for almost every $k$.
One possible representation of basis vectors $w^\eps_+$ and $w^\eps_-$ such that $E^\eps_+= \text{span}\{ w^\eps_+\}$ and $E^\eps_-= \text{span} \{ w^\eps_-\}$ is given by 
\begin{equation*}
w^\eps_\pm = \begin{pmatrix} 2 \epsi \beta \\ \delta-\nu - \epsi \alpha+ \mu \pm \Omega^\eps \end{pmatrix} ,
\end{equation*} 
where  $\Omega^\eps := \sqrt{(\delta-\nu - \epsi \alpha+\mu)^2 + 4 \epsi \beta \gamma}$.
We recall that the corresponding eigenvalues are given by
$$ 
\lambda^\eps_{\pm}(k) := \frac{1}{2}\left(\epsi \alpha - \mu + \delta -\nu - 8\pi^2 |k|^2  \mp \Omega^\eps \right) \in \R.
$$
From this we deduce that the attracting invariant slow manifold is given by $w^\eps_-$.
\end{proof} 

Since the eigenvectors completely determine the behaviour of system (\ref{linear system fourier}) we desire to compute $\lim_{\eps\to 0} w^\eps_\pm$.
In order for the eigenspaces to maintain a finite basis, even in the limit as $\eps \to 0$ we scale $w^\eps_\pm$ with $\eps$.
Now, we can take the limit and obtain
\begin{align*}
\eps w^\eps_\pm &= \begin{pmatrix} 2 \beta \\ \eps \delta - \alpha \pm \eps \Omega^\eps
\end{pmatrix} = \begin{pmatrix} 2 \beta \\ \eps (\delta-\nu +\mu) - \alpha \pm \sqrt{(\eps \delta -\eps \nu +\eps \mu - \alpha)^2 + 4 \eps \beta \gamma} \end{pmatrix}\\
&\to \begin{pmatrix} 2 \beta \\ - \alpha \pm |\alpha|
\end{pmatrix} := w_\pm^0
\end{align*}
as $\eps \to  0$. 
That is we obtain
$$  
w_-^0 = \begin{pmatrix} 2 \beta \\ - 2 \alpha \end{pmatrix}, \quad  w_+^0 = \begin{pmatrix} 2 \beta \\ 0 \end{pmatrix}.
$$
Thus, $w^0_-$ spans a linear subspace corresponding to 
\begin{equation*}
C_0^{\F}:= \{ (x,y) \in \R^2 : \alpha x + \beta y =0\} \subset \R^2,
\end{equation*}
which is the critical manifold of system (\ref{linear system fourier}) in the limit as $\eps \to 0$.
We have therefore found a family of invariant manifolds 
$C^{\F}_\eps:= E_+^\eps$, which is one eigenspace of $M_k^\eps$, such that $C^{\F}_\eps \to C^{\F}_0$ for $\eps \to 0$ for almost every $k$. 

%
%
This exact finite dimensional relation can be shown to also hold infinite dimensions.
Due to the identical structure of the solutions of (\ref{linear system}) and (\ref{linear complete system fourier}) these systems have the same linear invariant manifolds.
We summarize this fact in the following lemma.
\begin{lemma} \label{lemma finite to infinite}
Let $a,b \in \R \backslash \{0\}$, $H:=H^2(\R^n)$ and $\eps$ small enough such that $\Omega^\eps \in \R \backslash \{ 0\} $. Then, the following two statements are equivalent.
\begin{itemize}
    \item[(1)] There exists an invariant manifold, called the slow manifold, $S_\eps\subset H\times H$ to system (\ref{linear system});
    \item[(2)] $C^{\F}_\eps:=\{ (x,y) \in \R^2 : (\eps (\delta-\nu +\mu) - \alpha - \eps \Omega^\eps)x - 2 \beta y = 0 \}$ is an invariant manifold of system (\ref{linear complete system fourier}) for almost every $k \in \R^n$.
\end{itemize}
%
\end{lemma}
\begin{proof}
Note that $(0,0)$ is a steady state in both systems. We can therefore exclude it in what follows, since we are only interested in nontrivial invariant manifolds. We only show that (2) implies (1), but the other direction works exactly the same. \\
Let $w:=(u,v) \in C^{\F}_\eps\backslash\{0\}$ and choose any $k \in \R^n$ such that we can construct the explicit solution via Fourier transforms. Denote $\phi^t_k(u,v)$ as the flow of this system (\ref{linear system fourier}). Due to invariance, we have that $\phi_k^t(u,v) \in C^{\F}_\eps$ for all $t>0$. Expressed in formulas this means   
\begin{align*}
&a [(\rho^\eps_1 \e^{\lambda^\eps_+(k) t} + \rho^\eps_2 \e^{\lambda^\eps_-(k) t}) u  + (\rho_3^\eps \e^{\lambda^\eps_+(k) t} + \rho^\eps_4 \e^{\lambda^\eps_- t(k)}) v ] \\
+ &b [ (\rho_5^\eps \e^{\lambda_+^\eps(k) t} + \rho_6^\eps \e^{\lambda_-^\eps(k)} t) u + (\rho_7^\eps \e^{\lambda^\eps_+(k) t} + \rho^\eps_8 \e^{\lambda^\eps_-(k) t}) v] = 0  
\end{align*}
for all $t>0$. Since $\lambda_+(k) \neq \lambda_-(k)$ it follows by equating coefficients that
\begin{align*}
(a \rho^\eps_1   + b \rho^\eps_5 ) u + (a \rho^\eps_3 + b \rho^\eps_7 ) v = 
(a \rho^\eps_2  + b \rho^\eps_6 ) u + (a \rho^\eps_4 + b \rho^\eps_8 ) v =0.
\end{align*}
Since we have $u = - \frac{b}{a} v$, so we can also write the above equation as
\begin{align*}
\left( - b \rho^\eps_1   - \frac{b^2}{a} \rho^\eps_5  + a \rho^\eps_3 + b \rho^\eps_7 \right) v =
\left(- b \rho^\eps_2  - \frac{b^2}{a} \rho^\eps_6  + a \rho^\eps_4 + b \rho^\eps_8 \right) v= 0.
\end{align*}
Since $w \neq (0,0)$ we can assume $v\neq0$. So we can deduce from the last equation that 
\begin{align*}
\left( - b \rho^\eps_1   - \frac{b^2}{a} \rho^\eps_5  + a \rho^\eps_2 + b \rho^\eps_7 \right) =
\left(- b \rho^\eps_2  - \frac{b^2}{a} \rho^\eps_6  + a \rho^\eps_4 + b \rho^\eps_8 \right) = 0.
\end{align*}
Therefore, for any $\tilde {v} \in H^2(\R^n)$ we also have
\begin{align*}
\left( - b \rho^\eps_1   - \frac{b^2}{a} \rho^\eps_5  + a \rho^\eps_3 + b \rho^\eps_7 \right) *\tilde {v} = 
\left(- b \rho^\eps_2  - \frac{b^2}{a} \rho^\eps_6  + a \rho^\eps_4 + b \rho^\eps_8 \right) *\tilde {v}= 0
\end{align*}
and by defining $\tilde {u}:=-\frac{b}{a} \tilde{v} \in H^2(\R^n)$ we can follow
\begin{align*}
(a \rho^\eps_1   + b \rho^\eps_5 ) *  \tilde {u} + (a \rho^\eps_3 + b \rho^\eps_7 ) *  \tilde {v} =
(a \rho^\eps_2  + b \rho^\eps_6 )* \tilde {u} + (a \rho^\eps_4 + b \rho^\eps_8 ) *  \tilde {v} = 0
\end{align*}
for all $(\tilde {u},\tilde {v}) \in S_0$. From this it follows that
\begin{align*}
&a [(\rho^\eps_1 L^\eps_+(x,t) + \rho^\eps_2  L^\eps_-(x,t)) *\tilde {u}  + (\rho_3^\eps  L^\eps_+(x,t) + \rho^\eps_4  L^\eps_-(x,t)) *\tilde {v} ] \\
+ &b [ (\rho_5^\eps  L^\eps_+(x,t) + \rho_6^\eps  L^\eps_-(x,t) *\tilde {u} + (\rho_7^\eps  L^\eps_+(x,t) + \rho^\eps_8  L^\eps_-(x,t)) *\tilde {v}] = 0  
\end{align*}
for all $t>0$. So we have shown that $S_\eps$ is invariant.
\end{proof}
Similarly, we obtain that $C_0:=\{ (u,v) \in \R^2: a u + b v = 0\} \subset \R^2$ and $S_0:=\{ (u,v) \in H^2(\R^n) \times H^2(\R^n): a u + b v = 0\} \subset H \times H$ are equivalent in the above sense.
In addition, from the above lemma it can be concluded that two maximal invariant manifolds of system (\ref{linear system}) are given by
$$ 
C^\pm_\eps:= \{ (u,v) \in H \times H : (\eps \delta - \alpha \pm \eps \Omega^\eps) u - 2 \beta v = 0 \}. 
$$
As in the finite dimensional case $C_\eps^-$ converges to the critical attracting manifold of the linear fast-reaction system (\ref{linear system}) as $\eps \to 0$ if $\alpha<0$. 
This along with the  other properties of slow manifolds can be summarized in the following result which is an extension of the classical Fenichel-Tikhonov theory.

\begin{theorem} \label{slow manifold proposition}
Let $\alpha<0$ and $\eps$ small enough such that $\epsi \alpha-\mu<0$ and $\Omega^\eps \in \R \backslash \{0\}$. Then the following is true
%
%
\begin{itemize}
    \item[i)] $C^-_\eps $ is the attracting slow manifold of the linear fast-reaction system \eqref{linear system} which is invariant under the semi flow $T_\eps$. Moreover, the slow manifold has the same regularity as the critical manifold;
    \item[ii)] The slow manifold $C^-_\eps$ converges to the critical manifold  $C_0$ as $\eps \to 0$. In addition, on bounded subsets the distance between these two objects is given by $d_H(C^-_\eps,C_0) = \mathcal{O}(\eps)$ with $d_H$ being the Hausdorff-distance;
       \item[iii)] The semi flow $T_\eps$ restricted to $C^-_\eps$ converges to the semi flow $T_0$ on $C_0$ as $\eps \to 0$.
\end{itemize}

\end{theorem}
\begin{proof}
i) The invariance follows from Lemma \ref{lemma finite to infinite} and regularity properties holds due to the linearity of the manifolds. \\ \\
ii) Since $\eps \delta - \alpha - \eps \Omega^\eps \to - 2 \alpha$ as $\eps \to 0$, the equation describing the attracting slow manifold $C_\eps^-$ converges to the equation describing $C_0$.
Next, we define the Hausdorff distance as
$$ 
d_H(C_0,C_\eps^-) := \inf \{ {\lambda \geq 0} : C_0 \subset C_\eps^- + B_\lambda(0), C_\eps^- \subset C_0 + B_\lambda(0)  \}.
$$
Here $+$ denotes the Minkowski sum between sets. 
That is for $U,V \subset H^2(R^n)$ we define 
$$ 
U + V :=\{u + v : u \in U, v \in V\} \subset H^2(R^n).
$$ 
We first want to show $C_\eps^- \subset C_0 + B_\lambda(0)$.
To this end, let $w^\eps=(u,v^\eps)\in C_\eps^-$ such that $\|u\|_{H^2(\mathbb{R}^n)}\leq M$. Due to the convergence of $C^-_\eps \to C_0$, we can find $\tilde {\eps}>0$ such that $ |  \frac{\tilde {\eps} \delta - \alpha -\tilde\eps \Omega^{\tilde {\eps}}}{2 \beta} + \frac{\alpha}{\beta} | < \eps^2$. 
Let  $w^0=(u,- \frac{\beta}{\alpha} u)\in C_0$.
Since we consider an element on the slow manifold we have $v^\eps= \frac{\tilde {\eps} \delta - \alpha -\tilde\eps \Omega^{\tilde {\eps}}}{2 \beta} u$ and thus we can compute
\begin{align*}
    \|w^\eps - w^0\|^2_{H^2(R^n) \times H^2(R^n)} :&= \|u - u \|^2_{H^2(R^n)} + \| - \frac{\beta}{\alpha} u - v^\eps \|^2_{H^2(R^n)} \\& = \left|   \frac{\tilde {\eps} \delta - \alpha -\tilde\eps \Omega^{\tilde {\eps}}}{2 \beta} + \frac{\alpha}{\beta}\right| \| u \|^2_{H^2(R^n)}\\
    &< \eps^2  \| u \|^2_{H^2(R^n)} \leq \eps^2 M^2 .
\end{align*}
Hence we have shown that $C_0 \subset C_\eps^- + B_\lambda(0)$ for $\lambda=\eps M=\mathcal{O}(\eps)$.
The other inclusion can be shown in a similar way.
\\ \\ 
iii) If $(u_0,v_0) \in C_\eps^-$ we know that the semi flow $T_\eps(t)(u_0,v_0) \in C_\eps^-$ for all $t>0$ due to i). Therefore, we can eliminate $u^\eps$ in the second equation of system \eqref{linear system} and obtain
\begin{align*}
    \partial_t v^\eps(t) = (\Delta - \nu I) v^\eps(t) + \bigg(\gamma \frac{2 \beta}{\eps \delta - \alpha - \eps \Omega^\eps} + \delta\bigg) v^\eps(t).
\end{align*} 
Passing to the limit we obtain $\lim_{\eps \to 0} v^\eps = v^0$, where $v^0$ is the second component of the semi flow $T_0$. Then, 
$$u^\eps = \frac{2 \beta}{\eps \delta - \alpha - \eps \Omega^\eps} v^\eps \to - \frac{\beta}{\alpha}  v^0 = h^0 (v^0) \text{ as } \eps \to 0.$$ 
\end{proof}

\section{Conclusion}

The first observation we make is that the general results presented in \cite{kuehn2023fast} also apply in this linear fast-reaction setting. 
However, the advantage of this work is that the slow manifold is explicitly constructed and the convergence to the critical manifold only depends on the separation of time scales parameter $\varepsilon$ but not on any spectral gap type parameter.
Moreover, the results presented here hold for any dimension $N\in \mathbb{N}$ and can also be applied to bounded domains with either zero Dirichlet or periodic boundary conditions.

Secondly, we want to compare the normal hyperbolicity assumptions used in this work.
When working on the level of Fourier transforms the PDE system turns into a two dimensional ODE in each of the Fourier coefficients. 
Here, we can apply the normal hyperbolicity condition that the eigenvalues of the linearized matrix in the fast direction have a non-zero real part.
After taking the inverse Fourier transform, this translates to the normal hyperbolicity condition in the abstract Banach space setting, presented in \cite{kuehn2023fast}.

Thirdly, we comment on the influence of the parameter $\alpha$.
For $\alpha<0$ the critical manifold is attracting normally hyperbolic and so is the slow manifold. 
In addition, the convergence of solutions and the convergence of the semi flow on the manifolds holds for any $t\in (0,T)$ for any $T>0$.
Now, since this is a linear system, the results would also hold for $\alpha>0$ leading to a repelling slow manifold. 

Lastly, note carefully that we could only disregard the loss of normal hyperbolicity in the Fourier domain for some modes $k$ in our proofs as the measure of these modes is zero so it is not visible for the Fourier inversion. For nonlinear systems with a spectral gap, a similar mechanism is making the existence of invariant manifolds possible~ \cite{kuehn2023fast}. Yet, if the failure of formal hyperbolicity occurs on a set of positive measure in the Fourier domain, we conjecture that it is no longer possible by standard methods to derive invariant manifolds. In particular, this emphasizes the role of a nonlinearity in contrast to the expository linear case we have discussed here: the nonlinearity provides a possibility for mixing of the Fourier modes so that a positive measure set in the Fourier domain lacks normal hyperbolicity.


\printbibliography

@ARTICLE{Engeletal1,
   author = "M. Engel and F. Hummel and C. Kuehn and N. Popovi{\'c} and M. Ptashnyk and T. Zacharis",
   title = "Geometric analysis of fast-slow {PDEs} with fold singularities",
   journal = "arXiv:2207.06134",
   pages = {1--41},
   year = 2022,
   }

@article{Andreini2000result,
author = "Andreini, A. and Kamenskiĭ, M. and Nistri, P.",
title = "A result on the singular perturbation theory for differential inclusions in Banach spaces",
volume = "15",
journal = "Topological Methods in Nonlinear Analysis",
number = "1",
pages ="1 -- 15",
year = "2000"
}

@article{gudovich2001tikhonov,
  title="A Tikhonov-type theorem for abstract parabolic differential inclusions in Banach spaces",
  author="Gudovich, A. and Kamenskiĭ, M. and Nistri, P.",
  journal="Discussiones Mathematicae, Differential Inclusions, Control and Optimization",
  volume="21",
  number="2",
  pages="207--234",
  year="2001"
}

@article{hilhorst1996fast,
  title="The fast reaction limit for a reaction-diffusion system",
  author="Hilhorst, D. and Van Der Hout, R. and Peletier, L.",
  journal="Journal of Mathematical Analysis and Applications",
  volume="199",
  number="2",
  pages="349--373",
  year="1996"
}

@article{iida2018review,
  title="A review on reaction--diffusion approximation",
  author="Iida, M. and Ninomiya, H. and Yamamoto, H.",
  journal="Journal of Elliptic and Parabolic Equations",
  volume="4",
  number="2",
  pages="565--600",
  year="2018"
}

@article{chow1988invariant,
  title="Invariant manifolds for flows in Banach spaces",
  author="Chow, S.-N. and Lu, K.",
  journal="Journal of Differential Equations ",
  volume="74",
  number="2",
  pages="285--317",
  year="1988"
}

@article{FOIAS1988309,
title = "Inertial manifolds for nonlinear evolutionary equations",
journal = "Journal of Differential Equations",
volume = "73",
number = "2",
pages = "309-353",
year = "1988",
author = "Foias, C. and Sell, G. and Temam, R."
}

@article{tikhonov1952systems,
  title="Systems of differential equations containing small parameters in the derivatives",
  author="Tikhonov, A. N.",
  journal="Matematicheskii sbornik",
  volume="73",
  number="3",
  pages="575--586",
  year="1952"
}

@article{jones1995geometric,
  title="Geometric singular perturbation theory",
  author="Jones, C.",
  journal="Dynamical systems",
  pages="44--118",
  year="1995"
}

@article{fenichel1979geometric,
  title="Geometric singular perturbation theory for ordinary differential equations",
  author="Fenichel, N.",
  journal="Journal of differential equations",
  volume="31",
  number="1",
  pages="53--98",
  year="1979"
}

@article{fenichel1971persistence,
  title="Persistence and smoothness of invariant manifolds for flows",
  author="Fenichel, N. and Moser, J.K.",
  journal="Indiana University Mathematics Journal",
  volume="21",
  number="3",
  pages="193--226",
  year="1971"
}

@article{bates2008approximately,
  title="Approximately invariant manifolds and global dynamics of spike states",
  author="Bates, P.W. and Lu, K. and Zeng, C.",
  journal="Inventiones mathematicae",
  volume="174",
  number="2",
  pages="355--433",
  year="2008",
  publisher="Springer"
}

@article{bates2000invariant,
  title="Invariant foliations near normally hyperbolic invariant manifolds for semiflows",
  author="Bates, P.W. and Lu, K. and Zeng, C.",
  journal="Transactions of the American Mathematical Society",
  volume="352",
  number="10",
  pages="4641--4676",
  year="2000"
}

@book{bates1998existence,
  title="Existence and persistence of invariant manifolds for semiflows in Banach space",
  author="Bates, P. W. and Lu, K. and Zeng, C.",
  volume="645",
  year="1998",
  publisher="American Mathematical Society, Providence RI"
}

@book{kuehn2015multiple,
  title="Multiple time scale dynamics",
  author="Kuehn, C.",
  volume="191",
  year="2015",
  publisher="Springer, Cham"
}

@book{evansPDE,
   title =     {Partial Differential Equations},
   author =    {Lawrence C. Evans},
   publisher = {AMS, Providence RI},
   year =      {2010},
   series =    {Graduate Studies in Mathematics},
   volume =    {19}
}

@article{kuehn2023fast,
  title={Fast Reactions and Slow Manifolds},
  author={Kuehn, C. and Sulzbach, J.-E.},
  journal={arXiv preprint arXiv:2301.09368},
  year={2023},
  pages={1--31}
}

@article{engel2021connecting,
  title="Connecting a direct and a Galerkin approach to slow manifolds in infinite dimensions",
  author="Engel, M. and Hummel, F. and Kuehn, C.",
  journal="Proceedings of the American Mathematical Society, Series B",
  volume="8",
  number="21",
  pages="252--266",
  year="2021"
}

@article{hummel2022slow,
  title="Slow manifolds for infinite-dimensional evolution equations",
  author="Hummel, F. and Kuehn, C.",
  journal="Commentarii Mathematici Helvetici",
  volume="97",
  number="1",
  pages="61--132",
  year="2022"
}

@article{engel2020blow,
  title="Blow-up analysis of fast-slow PDEs with loss of hyperbolicity",
  author="Engel, M. and Kuehn, C.",
  journal="arXiv preprint arXiv:2007.09973",
  year="2020",
  pages="1--37"
}





\end{document}